\documentclass[a4paper,11pt]{article}
\usepackage[utf8]{inputenc}
\usepackage[T1]{fontenc}
\usepackage[nomath]{lmodern}
\usepackage{subcaption}
\usepackage{mathtools}
\usepackage{amsthm}
\usepackage{tikz}
\usepackage{mathrsfs,amssymb}  
\usepackage{enumitem}
\usepackage{fullpage}
\usetikzlibrary{calc}

\newtheorem{theorem}{Theorem}

\newtheorem{remark}[theorem]{Remark}
\newtheorem{claim}[theorem]{Claim}
\newtheorem{proposition}[theorem]{Proposition}
\newtheorem{lemma}[theorem]{Lemma}

\newtheorem{question}[theorem]{Question}

\newcommand*{\C}{\mathcal{C}}

\DeclareMathOperator{\vs}{vs}
\DeclareMathOperator{\ivs}{ivs}
\newcommand*{\ceil}[1]{\mathopen{}\left\lceil #1\right\rceil\mathclose{}}
\newcommand*{\bceil}[1]{\bigl\lceil #1\bigr\rceil}
\newcommand*{\abs}[1]{\lvert #1\rvert}
\newcommand{\ie}{i.e.\ }
\newcommand*{\ceilfrac}[2]{\mathopen{}\left\lceil\frac{#1}{#2}\right\rceil\mathclose{}}

\newenvironment{poc}{\begin{proof}[Proof of Claim]}{\end{proof}}

\title{When removing an independent set is optimal for reducing the chromatic number}
\author{
Stijn Cambie\thanks{Department of Computer Science, KU Leuven Campus Kulak Kortrijk, Belgium. E-mail: \texttt{stijn.cambie@hotmail.com}. Supported by the UK Research and Innovation Future Leaders Fellowship MR/S016325/1 and the Institute for Basic Science (IBS-R029-C4).}
\and John Haslegrave\thanks{Mathematical Institute, University of Oxford, UK. E-mail: \texttt{j.haslegrave@cantab.net}. Supported by the UK Research and Innovation Future Leaders Fellowship MR/S016325/1 and by the  European Research Council under the European Union's Horizon 2020 research and innovation programme (grant agreement no.\ 883810).}
\and
Ross J. Kang\thanks{Korteweg--de Vries Institute for Mathematics, University of Amsterdam, The Netherlands. 
Email: \texttt{r.kang@uva.nl}. Supported by a Vidi grant (639.032.614) of the Netherlands Organisation for Scientific Research (NWO).}
}

\date{\today}

\begin{document}

\maketitle

\begin{abstract}
    How large must the chromatic number of a graph be, in terms of the graph's maximum degree, to ensure that the most efficient way to reduce the chromatic number by removing vertices is to remove an independent set? By a reduction to a powerful, known stability form of Brooks' theorem, we answer this question precisely, determining the threshold to within two values (and indeed sometimes a unique value) for graphs of sufficiently large maximum degree. 
\end{abstract}

\section{Introduction}
Given a graph $G$, what is the minimum number of vertices one must remove in order to reduce its chromatic number? This natural invariant was introduced by Akbari et al.~\cite{ABKM22}, who called it the \textit{chromatic vertex stability number} of $G$, denoted $\vs_\chi(G)$, by analogy with the chromatic edge stability number defined by Staton \cite{Staton80}. There is a wide range of previous research on related concepts, such as $\chi$-critical graphs (see e.g.~\cite{BMR05}), reducing the chromatic number via other operations (see e.g.~\cite{King11}), or reducing other parameters by vertex deletion (see e.g.~\cite{PPR16}).

An obvious way to reduce the chromatic number is to start with an optimal colouring, and remove all vertices of one colour. This is equivalent to finding an independent set of vertices whose removal reduces the chromatic number, and the size of the smallest such set is the \textit{independent chromatic vertex stability number}, $\ivs_\chi(G)$. Akbari et al.~\cite{ABKM22} investigated under what circumstances we can guarantee that this is the most efficient approach, \ie that the chromatic vertex stability and independent chromatic vertex stability numbers are equal. 
They found that $\vs_\chi(G)=\ivs_\chi(G)$ for any graph $G$ with $\chi(G)\geq \Delta(G)$, but that there are examples $G$ for which $\vs_\chi(G)<\ivs_\chi(G)$ and $\chi(G)=\frac{\Delta(G)+1}{2}$. They defined the threshold function $f(\Delta)$ as the smallest quantity such that, for any graph $G$ of maximum degree $\Delta$, it must hold that $\vs_\chi(G)=\ivs_\chi(G)$ provided that $\chi(G)\geq f(\Delta)$. 
They asked the following, and felt the answer should be yes, \ie they expected their lower bound on $f$ to be best possible.

\begin{question}[{\cite[Problem 3.2]{ABKM22}}]\label{problem-half}
Is it true that $f(\Delta) \le \frac{\Delta}2+1$, \ie if $G$ is a graph with $\chi(G) \ge
\frac{\Delta(G)}2 + 1$, does it then always hold that $\vs_\chi(G) = \ivs_\chi(G)$?
\end{question}

We answer Question~\ref{problem-half} in the negative. Moreover, we settle the problem of determining $f(\Delta)$ in a strong sense. In fact, we show that it is the upper bound of Akbari et al.~which is close to best possible, in the sense that $f(\Delta)$ is always close to $\Delta$. Further, we determine $f(\Delta)$ almost exactly for every sufficiently large $\Delta$, showing that it takes one of at most two possible values. Our main results on $f$ are summarised as follows.
\begin{theorem}\label{thm:summary}
    As above, let $f(\Delta)$ be the smallest quantity such that, for any graph $G$ of maximum degree $\Delta$, it must hold that $\vs_\chi(G)=\ivs_\chi(G)$ provided that $\chi(G)\geq f(\Delta)$. Then, writing $k_{\Delta}=\left\lfloor \sqrt{\Delta+\frac14}-\frac 32 \right\rfloor$, the following statements hold.
    \begin{enumerate}
        \item $f(\Delta)=\Delta$ if $3\le\Delta\le10$.
        \item $f(\Delta) \ge \Delta+1-k_{\Delta}$ if $\Delta> 10$ and moreover there are infinitely many values for which $f(\Delta)\ge\Delta+2-k_{\Delta}$.
        \item $f(\Delta) \le \Delta+2-k_{\Delta}$ for all $\Delta$ sufficiently large.
    \end{enumerate}
\end{theorem}
A schematic for our main, elementary lower-bound construction is given in Figure \ref{fig:constructionD-o(D)}.

\begin{figure}[ht]
    \centering
\begin{tikzpicture}[thick]
\foreach \x in {-4,0,4}{\draw (\x,-.25) circle (1.25);}

\foreach \x in {0,60,...,300}{
	\foreach \y in {60,120,180}{\draw (\x:0.75) -- (\x+\y:0.75);}}
\foreach \x in {-4,4}{
	\draw (\x-0.5,0.5) -- (\x+0.5,0.5) -- (\x-0.5,-0.5) -- (\x+0.5,-0.5) -- cycle;
	\draw (\x-0.5,0.5) -- (\x-0.5,-0.5);
	\draw (\x+0.5,0.5) -- (\x+0.5,-0.5);}

\foreach \w in {-2,2}{
	\foreach \x in {-1,0,1}{
		\foreach \y in {-0.5,0.5}{
			\foreach \z in {0.5,-0.5}{
				\draw[orange] (\w,\x) -- (2*\w+\y,\z);}}}}
\draw [orange] (120:0.75) -- (-2,1) to[out=0, in=150] (60:0.75);
\draw [orange] (180:0.75) -- (-2,0) -- (240:0.75);
\draw [orange] (0:0.75) -- (-2,-1) to[out=0, in=-150] (300:0.75);
\draw [orange] (60:0.75) -- (2,1) to[out=180, in=30] (120:0.75);
\draw [orange] (0:0.75) -- (2,0) -- (300:0.75);
\draw [orange] (180:0.75) -- (2,-1) to[out=180, in=-30] (240:0.75);

\foreach \x in {0,60,...,300}{\filldraw (\x:0.75) circle (0.075);}
\foreach \x in {-4,4}{
	\node[anchor=north] at (\x,-0.75) {$K_{\chi-2}$};
	\foreach \y in {-0.5,0.5}{
		\foreach \z in {-0.5,0.5}{\filldraw (\x+\y,\z) circle (0.075);}}}
\foreach \x in {-2,2}{
	\foreach \y in {-1,0,1}{\filldraw (\x,\y) circle (0.075);}}
\node[anchor=north] at (0,-0.75) {$K_\chi$};
\end{tikzpicture}
    \caption{Construction for Proposition~\ref{prop:chi+sqrtchi_bound}}
    \label{fig:constructionD-o(D)}
\end{figure}
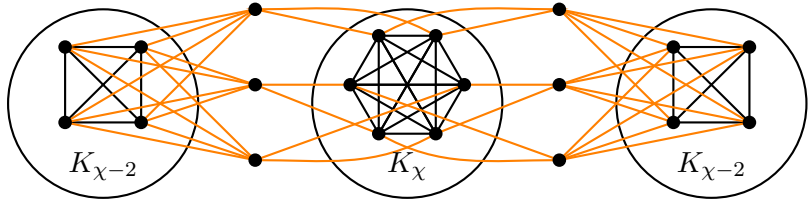

The proof of the upper bound of Akbari et al.\ makes use of Brooks' theorem, that for $\Delta\geq 3$ a graph $G$ of maximum degree $\Delta$ satisfying $\chi(G)=\Delta+1$ must contain a clique of order $\chi(G)$. In proving our upper bound, we use a deep result of Molloy and Reed \cite{MR14} which may be regarded as a stability version of Brooks' theorem: it says that for large $\Delta$ a chromatic number sufficiently close to this upper bound must also be caused by some small subgraph.

A natural approach to this problem, implicit in the lower-bound examples of Akbari et al., is to consider graphs $G$ where the essential difficulty in reducing the chromatic number lies in destroying all cliques of order $\chi(G)$, \ie $G$ for which $\omega(G)=\chi(G)$, the clique vertex stability number and chromatic vertex stability number are equal, but the independent clique vertex stability number is larger. Here we have defined the \textit{clique vertex stability number} $\vs_\omega(G)$ and \textit{independent clique vertex stability number} $\ivs_\omega(G)$ analogously, as the size of the smallest (independent) vertex set whose removal reduces the clique number of $G$. We give two general constructions based on this approach, both attaining $\chi(G)\approx \frac23\Delta(G)$; by a result of King~\cite{King11}, this is best possible. While our first construction of this form is simpler, the second is related to the satisfiability problem and so may be of independent interest.

We prove our upper bound in Section \ref{sec:preciseresult}, and give a construction attaining it (up to a difference of $1$) in Section \ref{sec:D=chi+sqrtchi}. We then consider the simultaneous clique and chromatic number version of the problem in Section \ref{sec:examples_with_vs_omega=vs_chi}, giving two essentially optimal constructions. We tie all the ends together for Theorem~\ref{thm:summary} in Section~\ref{sec:wrapup}.

\section{A precise upper bound for the threshold function}\label{sec:preciseresult}

Let $k_{\Delta}=\left\lfloor \sqrt{\Delta+\frac14}-\frac 32 \right\rfloor,$ \ie the maximum integer $k$ for which $(k+1)(k+2) \le \Delta$. In this section, we prove the following, which is our main result.
\begin{theorem}\label{thr:main}
    Let $G$ be a graph with sufficiently large maximum degree $\Delta$ and chromatic number $\chi(G)=c+1,$ for some $c \ge \Delta+1-k_{\Delta}.$
    Then $\vs_{\chi}(G)=\ivs_{\chi}(G)$.
\end{theorem}

We now proceed to the proof of Theorem \ref{thr:main}. We say that a graph $H$ is \textit{$(c+1)$-critical} if $\chi(H)=c+1$ and $\chi(H\setminus v)= c$ for every $v \in V(H)$.
Molloy and Reed~\cite{MR14} proved the following.

\begin{theorem}[\cite{MR14}]\label{thr:MR14}
    Let $G$ be a graph with sufficiently large maximum degree $\Delta$ and chromatic number $\chi(G)=c+1,$ for some $c \ge \Delta+1-k_{\Delta}.$
    Then $G$ contains a subgraph $H$ such that 
    \begin{itemize}
        \item $\lvert H \rvert \le \Delta+1$, and
        \item $H$ is $(c+1)$-critical.
    \end{itemize}
\end{theorem}
Note that it follows that every $(c+1)$-critical subgraph of $G$ has order at most $\Delta+1$, since otherwise applying the theorem to that subgraph gives a smaller $(c+1)$-critical subgraph, contradicting criticality.

We prove the following proposition.
\begin{proposition}\label{prop:smallcomponents}
    Let $G$ be a graph with sufficiently large maximum degree $\Delta$ and chromatic number $\chi(G)=c+1,$ for some $c \ge \Delta+1-k_{\Delta}.$
    Let $\C$ be a connected subgraph of $G$, which is the union of $(c+1)$-critical subgraphs $H_1,\ldots,H_r$ for some $r$.
    Then $\lvert \C \rvert \le \Delta+1+k_{\Delta}$.
\end{proposition}
\begin{proof}
    First we prove that the intersection of two $(c+1)$-critical subgraphs is large.
    \begin{claim}\label{clm:c-kneighb}
        Let $H_1$ and $H_2$ be two intersecting $(c+1)$-critical subgraphs of $G$.
        Then every vertex in $H_2 \setminus H_1$ has at least $c+1-k_{\Delta}$ neighbours in $H_1$.
    \end{claim}
    \begin{poc}
        A $(c+1)$-critical graph has minimum degree $\delta \ge c,$ since otherwise a $c$-colouring of the graph obtained by deleting a minimum-degree vertex may be extended to the whole graph.
        A vertex $v$ in $H_1 \cap H_2$ has at most $k_{\Delta}$ non-neighbours in $H_1$ and at most $k_\Delta$ non-neighbours in $H_2$. Together with the at most $\Delta$ neighbours of $v$ in $H_1 \cup H_2$, this implies that $\lvert H_1 \cup H_2 \rvert \le \Delta+1 + 2k_{\Delta},$ while $\lvert H_1 \rvert, \lvert H_2 \rvert \ge c+1.$
        Now $\lvert H_1 \cap H_2 \rvert = \lvert H_1 \rvert+ \lvert H_2 \rvert -\lvert H_1 \cup H_2 \rvert \ge \Delta+3-4 k_{\Delta}.$
        For the remainder of this proof, let $y=\lvert H_1 \cap H_2 \rvert$, $x=\lvert H_2 \setminus H_1 \rvert$ and $k=k_{\Delta}.$
        Note that the previous arguments give that $y> x$ provided that  $\Delta>8 k_{\Delta}$, which holds if $\Delta$ is sufficiently large.
        Since for every vertex $v \in H_1\cap H_2$ we have $\deg_{H_1}(v)\ge c$ and $\deg_G(v) \le \Delta$, there are at most $\Delta-c \le k$ neighbours of $v$ in $H_2 \setminus H_1$.
        On the other hand, every vertex in $H_2 \setminus H_1$ has at least $c-x$ neighbours in $H_2 \cap H_1$.
        Double counting of the edges between $H_2 \setminus H_1$ and $H_2 \cap H_1$ implies that
        $ky\ge x(c-x).$
        Since $x+y =\lvert H_2 \rvert \le \Delta +1\le c+k$, we have that 
        $k(c+k-x) \ge x(c-x)$, which is equivalent to
        $k^2\ge (x-k)(c-x).$
        Letting $g(x)=(x-k)(c-x)$ and noting that this is a concave parabola with $g(k+1)=g(c-1)=c-k-1>k^2$, we conclude that $x \le k$ (as $x\ge c$ is impossible for $\Delta$ sufficiently large).
        Since $\delta(H_2)\ge c$, the statement of the claim follows.
    \end{poc}
    Next we deduce that all the $(c+1)$-critical subgraphs in each component pairwise intersect.
        \begin{claim}\label{clm:intersect}
            The subgraphs $H_i$ and $H_j$ intersect for each $1\leq i,j\leq r$.
        \end{claim}
    \begin{poc}
    Suppose that some pair of subgraphs do not intersect. Since $\C$ is connected, there is a sequence of subgraphs such that every consecutive pair intersect but the first and last do not intersect. By choosing a minimal such sequence, we may assume $H_1$ and $H_3$ do not intersect but each intersect $H_2$. Choose $v\in H_2\cap H_3$. By Claim \ref{clm:c-kneighb}, $v$ has at least $c+1-k_{\Delta}$ neighbours in $H_1$, and since $H_3$ is $(c+1)$-critical it has at least $c$ neighbours in $H_3$. Thus $\deg_{\mathcal C}(v)\geq 2c+1-k_{\Delta}>\Delta,$ a contradiction.
    \end{poc}
    Now assume to the contrary that $\lvert \C \rvert \ge \Delta+2+k_{\Delta}$.
    Then we can pick $k_{\Delta}+1$ vertices in $\C$ not belonging to a fixed $(c+1)$-critical graph $H_1$, such that each of these $k_{\Delta}+1$ vertices belongs to some $(c+1)$-critical graph $H_i$ (different vertices may belong to different $H_i$), where each $H_i$ intersects $H_1$ by Claim \ref{clm:intersect}.
    By Claim~\ref{clm:c-kneighb} and the fact that $\delta(H_1)\ge c,$ the average degree of a vertex in $H_1$ has to be at least
    \[c+\frac{(k_{\Delta}+1)(\Delta-2k_{\Delta}+1)}{\Delta+1}\ge\Delta+1-k_{\Delta}+(k_{\Delta}+1)-\frac{2(k_{\Delta}+1)k_{\Delta}}{\Delta+1}>\Delta,\]
    which is a contradiction.
\end{proof}

We use the last proposition and the following result of Haxell~\cite{Hax01} to prove our main result.

\begin{theorem}[{\cite{Hax01}}]\label{thr:Haxell}
For a positive integer $k$, let $G$ be a graph with vertices partitioned
into $r$ cliques of size $\ge 2k$. If every vertex has at most $k$ neighbours outside its
own clique, then $G$ contains an independent set of size $r$.
\end{theorem}

\begin{proof}[Proof of Theorem \ref{thr:main}]
    By Theorem~\ref{thr:MR14}, in order to reduce the chromatic number it suffices to remove a vertex from every $(c+1)$-critical subgraph $H$ of size at most $\Delta+1$.
    By Proposition~\ref{prop:smallcomponents} we know the union of these form a subgraph $G'$ whose connected components are of small size, \ie each having at most $\Delta+1+k_\Delta\leq 2(c+1-k_{\Delta})$ vertices (since $\Delta$ is sufficiently large).
    If there are $r$ such components, we know that $\vs_{\chi}(G)\ge r.$
    Now, for an arbitrary $(c+1)$-colouring of $G$, every component $\C$ of $G'$ has at least $2 k_{\Delta}$ vertices that are coloured with a colour that appears exactly once on $\C,$ since all $c+1$ colours must be used on $\C$, and so otherwise it would have at least $2(c+1-(2 k_\Delta-1))+2 k_\Delta-1>2(c+1-k_\Delta)$ vertices, a contradiction.
    Every vertex in $\C$ has at least $c$ neighbours in the component and as such at most $\Delta-c \le k_{\Delta}-1$ neighbours outside of this component.
    For every $1\le i \le r$, we pick a vertex set $S_i$ consisting of $2 k_{\Delta}$ vertices in component $\C_i$, each of them coloured by a colour appearing only once on $\C_i$.
    Let $S= \cup_{i=1}^r S_i$ and consider the graph $G[S]$.
    Then by Theorem~\ref{thr:Haxell} (adding edges within each set $S_i$ if necessary) we can find an independent set $I$ of size $r$ in $G[S]$, containing one vertex $v_i$ of every $S_i$.
    Now $\chi(G \setminus I) =c=\chi(G)-1$ since no $(c+1)$-critical graph of size at most $\Delta+1$ is a subgraph of $G\setminus I.$
    This implies that $\ivs_{\chi}(G)\le r$ and more precisely we conclude that $\ivs_{\chi}(G)=\vs_{\chi}(G)=r.$
\end{proof}


\section{A corresponding lower bound construction}\label{sec:D=chi+sqrtchi}

\begin{proposition}\label{prop:chi+sqrtchi_bound}
    For every integer $\chi \ge 4,$ there exists a graph $G$ with $\omega(G)=\chi(G)=\chi$ and $\Delta=\chi+\max\left\{1, \ceilfrac{ \chi}{\ceil{\sqrt \chi}}  -2\right\}$ for which $\vs_\chi(G)<\ivs_\chi(G)$.
\end{proposition}

\begin{proof}
    We construct $G$ by first taking the disjoint union of a copy of $K_{\chi}$ and two copies of $K_{\chi-2}\vee aK_1$, where $a=\ceil{\sqrt \chi}$ and $\vee$ denotes the graph join (that is, the disjoint union of two graphs together with all edges between them). Denote the vertices of the two copies of $aK_1$ by $u_1,\ldots,u_a$ and $w_1,\ldots,w_a$.
    Then we partition the vertices of $K_{\chi}$ into $a$ non-empty subsets $V_1, V_2, \ldots, V_a$, each of them containing at most $\ceil{\frac{\chi}{a}}$ vertices, and for each $i$ add edges from $u_i$ to every vertex in $V_i$, and from $w_i$ to every vertex in $V_i$.
The construction is illustrated in Figure~\ref{fig:constructionD-o(D)}.
    We now verify that this construction works for all values of $\chi,$ by checking all of the promised conditions. 
    The $\omega(G)=\chi$ condition is a trivial one.
    
    \textbf{Condition $\chi(G)=\chi$}:
    Since $K_{\chi} \subset G,$ we have $\chi(G) \ge \chi.$
    In the other direction, first colour the copy of $K_\chi$ with $\chi$ colours.
    Now pick two colours $c_1$ appearing in $V_1$ and $c_2$ in $V_2$.
    All vertices in the copies of $aK_1$ can be coloured with either $c_1$ or $c_2$ (sometimes both are possible).
    The copies of $K_{\chi-2}$ can now be coloured with the $\chi-2$ colours different from $c_1$ and $c_2.$
    
    \textbf{Condition $\vs_\chi(G)=2$}:
    First we prove that $\chi(G\setminus v)=\chi$ for every $v \in G$.
    If $v$ does not belong to the copy of $K_\chi$, then $\chi \le \chi(G \setminus v) \le \chi(G)=\chi$.
    If $v \in K_\chi$, then in neither copy of $K_{\chi-2}\vee aK_1$ is a vertex removed. 
    If it were possible to colour $G \setminus v$ with $\chi-1$ colours, the vertices in each copy of $aK_1$ would be all coloured with one colour, being different from the ones used in the corresponding copy of $K_{\chi-2}$.
    But the central clique which has become a $K_{\chi-1}$ needs all $\chi-1$ colours as well, so at least one of the copies of $aK_1$ is adjacent to a vertex with the same colour, a contradiction. And so indeed $\chi(G\setminus v)=\chi$.
    This implies that $\vs_\chi(G)\ge 2.$
    
    For two vertices $u$ and $v$ from the central clique $K_\chi$, we have $\chi(G\setminus\{u,v\})=\chi-1.$
    For this, note that we can colour the two copies of $aK_1$ with one colour and the three copies of $K_{\chi -2}$ with the remaining $\chi-2$ other colours.
    So we conclude that $\vs_\chi(G)= 2.$
    
    \textbf{Condition $\ivs_\chi(G)>2$}:
    On the other hand, if we take an independent set $S$ of size $2$, we have $\chi(G\setminus S)=\chi$.
    If $S$ does not contain a vertex of the copy of $K_\chi$, this is clear.
    Otherwise, there is a copy of $aK_1\vee K_{\chi -2}$ from which no vertex is removed. As when verifying the condition $\vs_\chi(G)>1$, a $(\chi-1)$-colouring of this subgraph must use a single colour on the $aK_1$, making it impossible to colour the remaining $\chi-1$ vertices from the central clique.
    This implies that $\ivs_\chi(G)>2.$
    
    \textbf{Condition on the maximum degree}:
    There are three types of vertices.
    A vertex in a copy of $K_{\chi -2}$ has degree $\chi-3+a.$
    A vertex in a copy of $aK_1$ has maximum degree at most $(\chi -2)+\ceil{\frac{\chi}{a}}.$
    Since $\ceil{\frac{\chi}{a}}\ge a-1$, this is at least $\chi-3+a.$
    A vertex in the copy of $K_{\chi}$ has degree $(\chi-1)+2=\chi+1.$
    This concludes the proof.
\end{proof}

\begin{remark}
    Note that one can do the same construction with $\ceil{\frac{\chi}{a}} -1$ copies of $K_{\chi-2}\vee aK_1$.
    This has the same maximum degree, and now $\vs_\chi(G)=2$, while $\ivs_\chi(G)=\ceil{\frac{\chi}{a}}$. So even for 
    $\chi(G)\geq(1-o(1))\Delta(G)$, the difference between $\ivs_\chi(G)$ and $\vs_\chi(G)$ can be unbounded.
\end{remark}

\section{Relations with clique stability}\label{sec:examples_with_vs_omega=vs_chi}
In this section we give two essentially optimal constructions with $\omega(G)=\chi(G)$ where both the chromatic vertex stability and the clique vertex stability differ from their independent versions. Note that, while in general $\ivs_\omega(G)$ might not exist, the condition $\omega(G)=\chi(G)$ ensures that it does, since removing a colour class of an optimal colouring reduces the chromatic number and so must also reduce the clique number.

King~\cite[Theorem 3]{King11} proved that for every graph $G$ with $\omega(G)>\frac23 (\Delta(G)+1)$, there is an independent set $I$ for which $\omega(G \setminus I)<\omega(G)$. Although he did not explicitly state it, his construction clearly shows that there do not exist smaller vertex sets $S$ for which $\omega(G \setminus S)<\omega(G)$. So stated differently, he proved the following theorem.

\begin{theorem}[\cite{King11}]
    If a graph $G$ satisfies $\omega(G)>\frac23
    (\Delta(G)+1)$, then $\ivs_\omega(G)$ exists and equals $\vs_\omega(G)$.
\end{theorem}

This theorem is sharp for the clique-blow-up of $C_5$, \ie the graph $G_k$ obtained by substituting every vertex of a $5$-cycle with a clique of size $k$, with all edges between cliques corresponding to adjacent vertices.
In that case $\omega(G_k)=2k$, $\Delta(G_k)=3k-1$ and $\chi(G_k)\geq\ceilfrac{5k}{2}$. To see the latter, we may repeatedly use five colours to colour two vertices from each clique, and, if $k$ is odd, colour the final five vertices with three additional colours. Since no colour can be used more than twice, at least this many colours are required.
For this graph, $\ivs_\omega(G_k)$ is undefined, while $\vs_\omega(G_k)=3$.
On the other hand, $\vs_\chi(G_k)=\ivs_\chi(G_k)$ equals $1$ or $2$ depending on the parity of $k$.

We give two constructions for which $\chi(G)=\omega(G)$ and such that $\vs_\chi(G)<\ivs_\chi(G)$, provided $\chi(G)< \frac23 (\Delta(G)+1)$.
\begin{figure}[ht]
    \centering
    \begin{tikzpicture}[thick]
\draw (-5,-1) -- (-5,-3);
\draw (-5,-2) -- (-2,-2);
\draw (-2,-1) -- (-2,-3);
\filldraw (-4,-2) circle (0.075);
\filldraw (-3,-2) circle (0.075);
\filldraw (-5,-1) circle (0.075);
\filldraw (-5,-2) circle (0.075);
\filldraw (-5,-3) circle (0.075);
\filldraw (-2,-1) circle (0.075);
\filldraw (-2,-2) circle (0.075);
\filldraw (-2,-3) circle (0.075);

\filldraw (3,-1) circle (0.075);
\filldraw (2,-2) circle (0.075);
\filldraw (4,-2) circle (0.075);
\filldraw (3,-3) circle (0.075);
\draw (3,-1) -- (2,-2) -- (3,-3) -- (4,-2) -- cycle;
\draw (3,-1) -- (3,-3);
\draw (2,-2) -- (4,-2);

\filldraw (2.5,0) circle (0.075);
\filldraw (3.5,0) circle (0.075);
\draw (2.5,0) -- (3,-1) -- (3.5,0) -- cycle;
\filldraw (2.5,-4) circle (0.075);
\filldraw (3.5,-4) circle (0.075);
\draw (2.5,-4) -- (3,-3) -- (3.5,-4) -- cycle;
\foreach \x in {1.5,2.5,3.5,4.5}{
	\filldraw (\x,1) circle (0.075);
	\filldraw (\x,-5) circle (0.075);
	\draw (2.5,0) -- (\x,1) -- (3.5,0);
	\draw (2.5,-4) -- (\x,-5) -- (3.5,-4);}
\draw (1.5,1) -- (2.5,1);
\draw (3.5,1) -- (4.5,1);
\draw (1.5,-5) -- (2.5,-5);
\draw (3.5,-5) -- (4.5,-5);

\filldraw (1,-1.5) circle (0.075);
\filldraw (1,-2.5) circle (0.075);
\draw (1,-1.5) -- (2,-2) -- (1,-2.5) -- cycle;
\filldraw (5,-1.5) circle (0.075);
\filldraw (5,-2.5) circle (0.075);
\draw (5,-1.5) -- (4,-2) -- (5,-2.5) -- cycle;
\foreach \x in {-.5,-1.5,-2.5,-3.5}{
	\filldraw (0,\x) circle (0.075);
	\filldraw (6,\x) circle (0.075);
	\draw (1,-1.5) -- (0,\x) -- (1,-2.5);
	\draw (5,-1.5) -- (6,\x) -- (5,-2.5);}
\draw (0,-.5) -- (0,-1.5);
\draw (0,-2.5) -- (0,-3.5);
\draw (6,-.5) -- (6,-1.5);
\draw (6,-2.5) -- (6,-3.5);
\end{tikzpicture}
    \caption{The construction of Proposition~\ref{prop:constr1} for $\Delta=3$ and $\Delta=6$. In the former, the two central vertices form the $K_k$ and each is adjacent to one vertex in a copy of $H$; in the latter, the four central vertices form the $K_k$ and each is adjacent to two vertices in a copy of $H$.}
    \label{fig:G_builtfromKkandHs}
\end{figure}
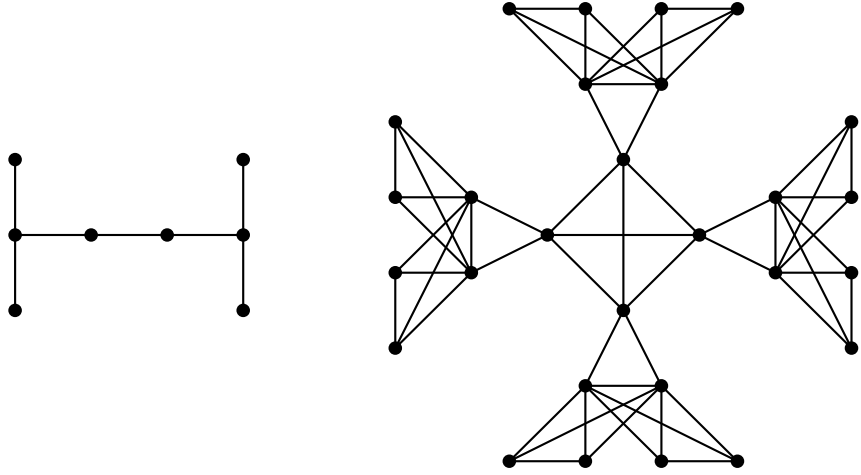

\begin{figure}[ht]
    \centering
\begin{tikzpicture}[thick]
\draw (0,0) circle (1);
\draw (18:2.5) -- (23:3) -- (13:3) -- cycle;
\draw (33:2.5) -- (30.5:3.5);
\foreach \x in {18,90,162,234,306}{
	\node at (\x:4.25) {$H$};
	\foreach \y in {72,144}{
		\draw (\x:0.75) -- (\x+\y:0.75);}
	\draw[orange] (\x:0.75) -- (\x:2.5);
	\filldraw (\x:2.5) circle (0.075);
	\foreach \y in {10,-10}{
		\draw (\x+\y:3) circle (1);
		\draw[orange] (\x:0.75) -- (\x+.5*\y:3);
		\filldraw (\x+1.5*\y:2.5) circle (0.075);
		\filldraw (\x+1.25*\y:3.5) circle (0.075);
		\filldraw (\x+.5*\y:3) circle (0.075);}
	\filldraw (\x:0.75) circle (0.075);}
\node at (0,-1.25) {$K_k$};
\node at (75:4.25) {$K_k$};
\node at (105:4.25) {$K_k$};
\node at (18:3.4) {$K_a$};
\node at (38:3) {$K_{k-a}$};
\draw (90:2.5) -- (95:3) -- (85:3) -- cycle;
\foreach \y in {10,-10}{
	\draw (90+1.5*\y:2.5) -- (90:2.5) -- (90+1.25*\y:3.5) -- cycle;
	\draw (90+1.5*\y:2.5) -- (95:3) -- (90+1.25*\y:3.5);
	\draw (90+1.5*\y:2.5) -- (85:3) -- (90+1.25*\y:3.5);}
\end{tikzpicture}
   \caption{Schematic of the construction used in Proposition~\ref{prop:constr1} }
    \label{fig:G_builtfromKkandHs_schematic}
\end{figure}
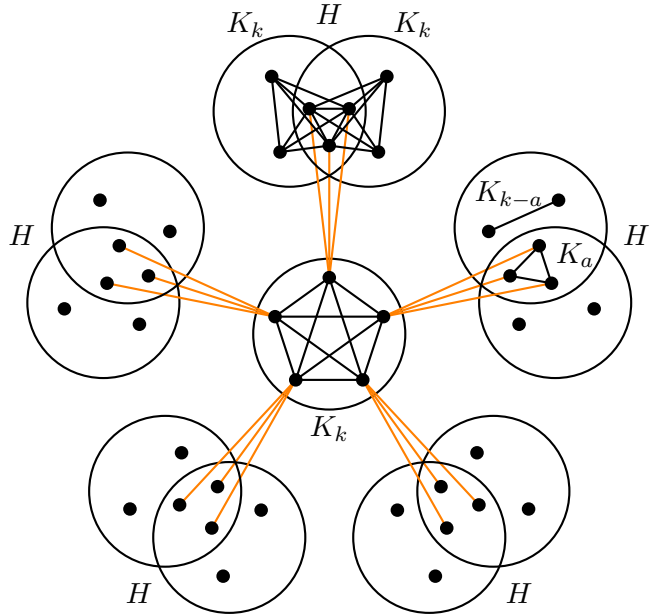

\begin{proposition}\label{prop:constr1}
    For every integer $\Delta \ge 2$, there exists a graph $G$ with $\Delta(G)=\Delta$ and $\omega(G)=\chi(G)=\bceil{\frac23 (\Delta+1)} -1$ such that $\vs_\chi(G)=\vs_\omega(G)<\ivs_\omega(G)=\ivs_\chi(G)$.
\end{proposition}
\begin{proof}
    Let $k=\bceil{\frac23 (\Delta+1)} -1$ and let $a=2k-\Delta $.
    Let $H$ be the union of two $K_k$'s such that they have $a$ vertices in common, \ie the intersection of the two $K_k$'s is a $K_a$.
    Now construct $G$ as follows. We take $k$ disjoint copies of $H$ and one additional copy of $K_k$.
    Connect each vertex of $K_k$ to the $a$ vertices belonging to both $K_k$s of a different copy of $H$.
    The resulting graphs are presented for $\Delta \in \{3,6\}$ in Figure~\ref{fig:G_builtfromKkandHs}, or in a more schematic way in Figure~\ref{fig:G_builtfromKkandHs_schematic}.
    Observe that $G$ satisfies $\Delta(G)=\max\{2k-a,k-1+a\}=\Delta.$
    Since $k-a\ge 1$, it is also not hard to see that $\chi(G)=\omega(G)=k.$
    Furthermore we have $\vs_\chi(G)=k+1,$ since we need to remove at least one vertex of the central $K_k$ and one vertex of each copy of $H$.
    If we choose the vertex to be one of the $a$ vertices belonging to both $K_k$s in the construction of $H$, it is also clear that after removing the $k+1$ vertices we end with a graph with chromatic number equal to $k-1.$
    On the other hand $\ivs_\chi(G)=k+2,$ since we cannot take an independent set with precisely one vertex of the $K_k$ and every $K_a \subset H$. We can, however, reduce the chromatic number by removing an independent set consisting of one vertex from the $K_k$, one vertex from the $K_a$ in all but one copy of $H$, and two non-adjacent vertices in the final copy of $H$.
\end{proof}

Our second example is based on SAT. The motivation is the standard proof of the NP-hardness of the independent set problem, which uses a graph-based reduction from SAT. Given a SAT instance $\mathcal I$ with $k$ clauses, the independence graph $G(\mathcal I)$ consists of one vertex for each literal in each clause, with two vertices being adjacent if either they are in the same clause or they correspond to complementary literals. The instance $\mathcal I$ is satisfiable if and only if $G(\mathcal I)$ has an independent set of order $k$; see \cite{Karp}. A simple example is given in Figure \ref{fig:SAT-simple}. 

\begin{figure}[ht]
	\centering
	\begin{tikzpicture}[thick]
		\foreach \x in {0,3,6}{
			\draw (\x,0) -- (\x+1.5,0) -- (\x+.75,{.75*sqrt(3)}) -- cycle;
		}
		\draw (0.75,{.75*sqrt(3)}) -- (6.75,{.75*sqrt(3)});
		\draw (1.5,0) -- (3,0);
		\draw (4.5,0) -- (6,0);
		\draw (0,0) to [out=-20, in =-160] (7.5,0);
		\foreach \x in {0,3,4.5,7.5}{
			\filldraw (\x,0) circle (.075);
		}
		\filldraw[fill=white] (1.5,0) circle (.075);
		\filldraw[fill=white] (6,0) circle (.075);
		\filldraw (.75,{.75*sqrt(3)}) circle (.075);
		\filldraw[fill=white] (3.75,{.75*sqrt(3)}) circle (.075);
		\filldraw (6.75,{.75*sqrt(3)}) circle (.075);
		\node [anchor=south] at (.75,{.75*sqrt(3)}) {$a$};
		\node [anchor=south] at (3.75,{.75*sqrt(3)}) {$\overline{a}$};
		\node [anchor=south] at (6.75,{.75*sqrt(3)}) {$a$};
		\node [anchor=south east] at (0,0) {$b$};
		\node [anchor=south west] at (1.5,0) {$c$};
		\node [anchor=south east] at (3,0) {$\overline{c}$};
		\node [anchor=south west] at (4.5,0) {$d$};
		\node [anchor=south east] at (6,0) {$\overline{d}$};
		\node [anchor=south west] at (7.5,0) {$\overline{b}$};
	\end{tikzpicture}
\caption{The independence graph for the SAT instance $(a\vee b\vee c)\wedge(\overline{a}\vee\overline{c}\vee d)\wedge(a\vee\overline{d}\vee\overline{b})$. White vertices give an independent set of order $3$, corresponding to the satisfying assignments where $a=d=0$ and $c=1$.}
\label{fig:SAT-simple}
\end{figure}
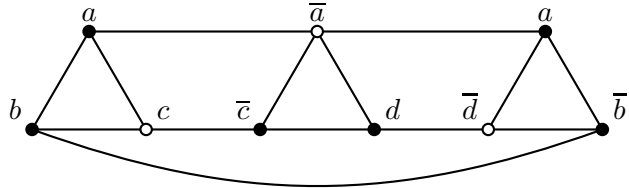

If all clauses have the same number $q$ of literals, then the independence graph consists of disjoint copies of $K_q$ with some edges between them, and the satisfiability of the instance depends on whether an independent set consisting of one vertex from each copy of $K_q$ can be selected. 
In order to relate this independent set problem to chromatic stability, we define the augmented independence graph $\widetilde G(\mathcal I)$, which is obtained from $G(\mathcal I)$ by adding two extra vertices for each clause, adjacent to all vertices from that clause, but not to each other. Thus each copy of $K_q$ is replaced by two copies of $K_{q+1}$ with $q$ vertices in common. The satisfiability of the instance determines whether the most efficient way to remove one vertex from each of these copies of $K_{q+1}$ is to remove an independent set.

The maximum degree of the augmented independence graph is determined by the largest number of times some literal occurs in the instance. We thus consider a restricted class of SAT where we control both the sizes of the cliques and the maximum degree. Define $p$-LIT $q$-SAT to be the class of SAT instances such that each clause contains exactly $q$ literals, and each literal appears at most $p$ times. Note that we permit the $q$ literals in a clause to contain the same literal more than once, up to the limit of $p$ total occurrences of that literal. The key question for our purposes is whether unsatisfiable instances exist.

\begin{proposition}\label{prop:SAT}For each $m\geq 2$ there exists an instance $\mathcal I$ of $m$-LIT $(2m-1)$-SAT which is not satisfiable and has $\chi(G(\mathcal I))=2m$.
\end{proposition}
Note that this is tight, in the sense that every instance of $m$-LIT $2m$-SAT is satisfiable. To see this, draw a bipartite multigraph with vertices corresponding to variables and clauses, and one edge between a variable and clause for every occurrence (positive or negated) of that variable in that clause. Thus each vertex corresponding to a clause meets $2m$ edges, whereas each vertex corresponding to a variable meets at most $2m$ edges. Any set $S$ of clauses meets $2m\abs S$ edges, which must go to at least $\abs S$ different variables, so Hall's theorem allows us to select a different variable adjacent to each clause, and hence to select one literal from each clause no two of which are complementary, yielding a satisfying assignment. (The argument above is due to Tovey \cite[Theorem 2.4]{Tov84}.)
\begin{proof}[Proof of Proposition \ref{prop:SAT}]Let $r=\ceil{\log_2m}$. We will inductively construct a sequence $\mathcal I_0,\ldots,\mathcal I_r$ such that $\mathcal{I}_i$ is an instance of $m$-LIT $(2m-2^{r-i})$-SAT which is not satisfiable and has $\chi(G(\mathcal I_i))\leq 2m$. Our construction will also ensure that $G(\mathcal I_i)$ contains a clique of order $2m$ for every $i\geq 1$. Thus $\mathcal I_r$ will satisfy the requirements of the theorem.

Take $\mathcal I_0$ to consist of two clauses, one consisting of $2m-2^r$ copies of $a$ and the other of $2m-2^r$ copies of $\overline a$. By definition of $r$ we have $2^{r-1}<m$, and so the clauses are non-empty and it is impossible to satisfy both of them. Each literal occurs $2m-2^r\leq m$ times, and $G(\mathcal I_0)$ is a clique on $4m-2^{r+1}\leq 2m$ vertices, so $\mathcal I_0$ has the required properties.

Now, for each $i=1,\ldots,r$, construct $\mathcal I_i$ from $\mathcal I_{i-1}$ as follows. For each clause $x$ of $\mathcal I_{i-1}$ we create two clauses $x^+,x^-$ of $\mathcal I_i$. To do this, choose a new variable $b_x$, and separate the $2m-2^{r-i+1}$ literals appearing in $x$ (arbitrarily) into two groups of $m-2^{r-i}$. Set $x^+$ to consist of $m$ copies of $b_x$ together with the first group, and set $x^-$ to consist of $m$ copies of $\overline{b_x}$ together with the second group. Since each literal from $\mathcal I_{i-1}$ appears the same number of times in $\mathcal I_i$, and each new literal appears $m$ times, it is indeed an instance of $m$-LIT $(2m-2^{r-i})$-SAT.

Suppose that a truth assignment satisfies $\mathcal I_i$. Then for each clause $x$ of $\mathcal I_{i-1}$, it satisfies both $x^+$ and $x^-$ and therefore must satisfy $x$. Since $\mathcal I_{i-1}$ was not satisfiable, neither is $\mathcal I_i$.

Finally, we show that $\chi(G(\mathcal I_i))\leq 2m$; since the graph contains a copy of $K_{2m}$ for every new variable $b_x$, we must also have equality. Take a colouring of $G(\mathcal I_{i-1})$ with at most $2m$ colours. Note that the subgraph of $G(\mathcal I_i)$ induced by literals appearing in $\mathcal I_{i-1}$ is a subgraph of $G(\mathcal I_{i-1})$, since any two such literals appearing in the same clause of $\mathcal I_i$ came from the same clause of $\mathcal I_{i-1}$. Consequently this is a proper partial colouring of $G(\mathcal I_i)$. We extend it to the new literals. For each pair $x^+,x^-$, write $C_x^+$ (respectively $C_x^-$) for the set of colours already used on $x^+$ (respectively $x^-$). Note that, since we started from a proper colouring of $G(\mathcal I_{i-1})$, we have $C_x^+\cap C_x^-=\varnothing$, and $\abs{C_x^+}=\abs{C_x^-}=m-2^{r-i}$. Assign the colours in $C_x^+$ to $m-2^{r-i}$ of the $\overline{b_x}$ literals, and the colours in $C_x^-$ to $m-2^{r-i}$ of the $b_x$ literals. Since the two sets were disjoint, this is still a proper partial colouring. Assign the $2m-2(m-2^{r-i})=2^{r-i+1}$ colours in $[2m]\setminus(C_x^+\cup C_x^-)$ to the $2^{r-i+1}$ remaining copies of $b_x$ and $\overline{b_x}$. Since each of these only has neighbours in $x^+$ or $x^-$, this gives a proper colouring of $G(\mathcal I_i)$.
\end{proof}
For example, with $m=4$ (and so $r=2$), the construction might proceed 
\begin{align*}\mathcal{I}_0&=(a\vee a\vee a\vee a)\wedge(\overline a\vee\overline  a\vee\overline  a\vee\overline  a),\\
\mathcal{I}_1&=(b\vee b\vee b\vee b\vee a\vee a)\wedge(\overline b\vee\overline b\vee\overline b\vee\overline b\vee a\vee a)\\
&\phantom{{}={}}\wedge(c\vee c\vee c\vee c\vee\overline a\vee\overline a)\wedge(\overline c\vee\overline c\vee\overline c\vee\overline c\vee\overline a\vee\overline a),\\
\mathcal{I}_2&=(d\vee d\vee d\vee d\vee b\vee b\vee a)\wedge(\overline d\vee\overline d\vee\overline d\vee\overline d\vee b\vee b\vee a)\\
&\phantom{{}={}}\wedge(e\vee e\vee e\vee e\vee \overline b\vee\overline b\vee a)\wedge(\overline e\vee\overline e\vee\overline e\vee\overline e\vee \overline b\vee\overline b\vee a)\\
&\phantom{{}={}}\wedge(f\vee f\vee f\vee f\vee c\vee c\vee \overline a)\wedge(\overline f\vee\overline f\vee\overline f\vee\overline f\vee c\vee c\vee \overline a)\\
&\phantom{{}={}}\wedge(g\vee g\vee g\vee g\vee \overline c\vee\overline c\vee\overline a)\wedge(\overline g\vee\overline g\vee\overline g\vee\overline g\vee \overline c\vee\overline c\vee\overline a).
\end{align*}
Note, however, that the splitting of each clause is arbitrary, and while at the first step there is only one way to do this, subsequently there are genuine options producing non-isomorphic constructions. The independence graph for this example is shown in Figure \ref{fig:long-SAT}.

\begin{figure}[ht]
\centering
\begin{tikzpicture}[thick]
\foreach \x in {0,3,7,10}{
\filldraw[draw=gray!75,fill=gray!75] ($(-110:1)+(\x,0)$) -- ($(110:1)+(\x,-3)$) -- ($(70:1)+(\x,-3)$) -- ($(-70:1)+(\x,0)$) -- cycle;
}
\foreach \x in {0,7}{
\filldraw[draw=gray!75,fill=gray!75] ($(-110:1)+(\x,0)$) -- ($(110:1)+(\x+3,-3)$) -- ($(70:1)+(\x+3,-3)$) -- ($(-70:1)+(\x,0)$) -- cycle;
\filldraw[draw=gray!75,fill=gray!75] ($(110:1)+(\x,-3)$) -- ($(-110:1)+(\x+3,0)$) -- ($(-70:1)+(\x+3,0)$) -- ($(70:1)+(\x,-3)$) -- cycle;
\filldraw[fill=gray!75,draw=gray!75] ($(90:1)+(\x,0)$) -- ($(90:1)+(\x+3,0)$) -- ($(-150:1)+(\x+3,0)$) -- ($(-30:1)+(\x,0)$) -- cycle;
\filldraw[fill=gray!75,draw=gray!75] ($(-90:1)+(\x,-3)$) -- ($(-90:1)+(\x+3,-3)$) -- ($(150:1)+(\x+3,-3)$) -- ($(30:1)+(\x,-3)$) -- cycle;
\filldraw[fill=gray!75] ($(90:1)+(\x,0)$) -- ($(170:1)+(\x,0)$) -- ($(-110:1)+(\x,0)$) -- ($(-70:1)+(\x,0)$) -- ($(-30:1)+(\x,0)$) -- ($(10:1)+(\x,0)$) -- ($(50:1)+(\x,0)$) -- cycle;
\filldraw[fill=gray!75] ($(90:1)+(\x+3,0)$) -- ($(10:1)+(\x+3,0)$) -- ($(-70:1)+(\x+3,0)$) -- ($(-110:1)+(\x+3,0)$) -- ($(-150:1)+(\x+3,0)$) -- ($(170:1)+(\x+3,0)$) -- ($(130:1)+(\x+3,0)$) -- cycle;
\filldraw[fill=gray!75] ($(-90:1)+(\x,-3)$) -- ($(-170:1)+(\x,-3)$) -- ($(110:1)+(\x,-3)$) -- ($(70:1)+(\x,-3)$) -- ($(30:1)+(\x,-3)$) -- ($(-10:1)+(\x,-3)$) -- ($(-50:1)+(\x,-3)$) -- cycle;
\filldraw[fill=gray!75] ($(-90:1)+(\x+3,-3)$) -- ($(-10:1)+(\x+3,-3)$) -- ($(70:1)+(\x+3,-3)$) -- ($(110:1)+(\x+3,-3)$) -- ($(150:1)+(\x+3,-3)$) -- ($(-170:1)+(\x+3,-3)$) -- ($(-130:1)+(\x+3,-3)$) -- cycle;
}
\draw[dashed] ($(10:1)+(3,0)$) -- ($(170:1)+(7,0)$) -- ($(-10:1)+(3,-3)$) -- ($(-170:1)+(7,-3)$) -- cycle;
\draw[dashed] ($(10:1)+(3,0)$) to[out=30, in=150] ($(10:1)+(10,0)$) -- ($(-10:1)+(3,-3)$) to[out=-30, in=-150] ($(-10:1)+(10,-3)$) -- cycle;
\draw[dashed] ($(170:1)+(0,0)$) to[out=30, in=150] ($(170:1)+(7,0)$) -- ($(-170:1)+(0,-3)$) to[out=-30, in=-150] ($(-170:1)+(7,-3)$) -- cycle;
\draw[dashed] ($(170:1)+(0,0)$) to[out=30, in=150] ($(10:1)+(10,0)$) -- ($(-170:1)+(0,-3)$) to[out=-30, in=-150] ($(-10:1)+(10,-3)$) -- cycle;
\end{tikzpicture}
\caption{The (unaugmented) independence graph $G(\mathcal I_2)$ for the example $4$-LIT $7$-SAT instance $\mathcal I_2$ given after the proof of Proposition \ref{prop:SAT}. Each shaded heptagon represents a clique corresponding to a clause, and other shaded sections are complete bipartite subgraphs. The dashed edges connect vertices corresponding to $a$ and $\overline{a}$.}
\label{fig:long-SAT}
\end{figure}
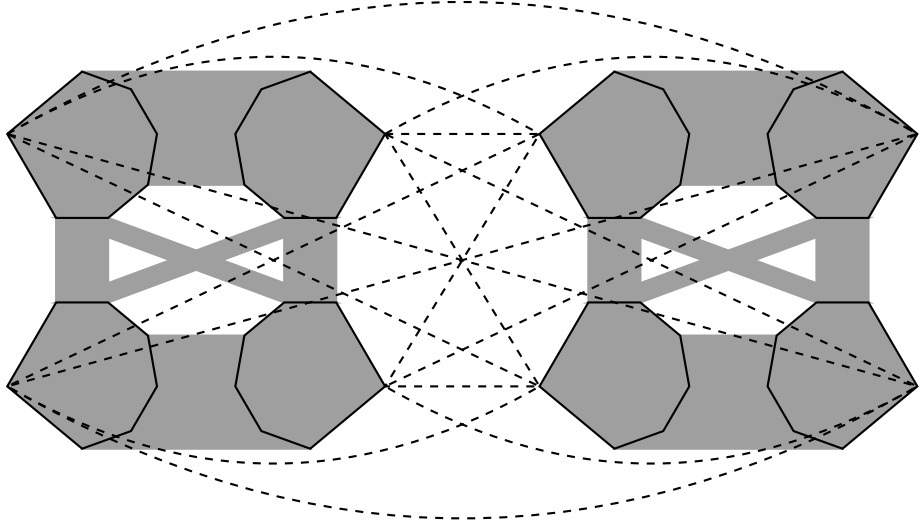
We next show that we may reduce the chromatic number of $G(\mathcal I)$ by removing one vertex from each clause. 
\begin{lemma}\label{lem:removal}Let $G$ be a $k$-chromatic graph and let $P$ be a partition of $V(G)$ such that each part induces a clique. Then we can choose a set $S$ consisting of one vertex from each part such that $G-S$ is $(k-1)$-colourable.\end{lemma}
\begin{proof}Colour $G$ with $k$ colours, and choose a highest-coloured vertex from each part to make up $S$. Since each part has at most one vertex of colour $k$, every vertex of colour $k$ is in $S$, so at most $k-1$ colours are used on $G-S$.\end{proof}
Finally, we show that the augmented independence graph has the required properties. An example for $2$-LIT $3$-SAT is shown in Figure \ref{fig:SAT}.
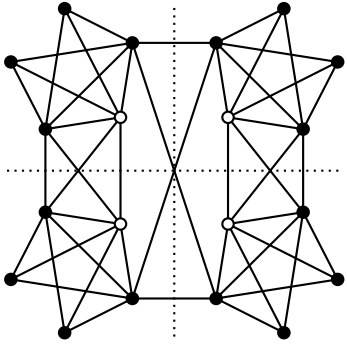
\begin{figure}[ht]
\centering
\begin{tikzpicture}[thick]
\draw (45:1) -- ++(-9:1) coordinate (a);
\draw (45:1) -- ++(99:1) coordinate (b);
\draw (135:1) -- ++(81:1) coordinate (c);
\draw (135:1) -- ++(189:1) coordinate (d);
\draw (225:1) -- ++(171:1) coordinate (e);
\draw (225:1) -- ++(-81:1) coordinate (f);
\draw (315:1) -- ++(-99:1) coordinate (g);
\draw (315:1) -- ++(9:1) coordinate (h);
\draw (45:1) -- (h) -- (a) -- (315:1) -- cycle;
\draw (135:1) -- (e) -- (d) -- (225:1) -- cycle;
\draw (b) -- (c) -- (g) -- (f) -- cycle;
\draw (a) -- ++(63:1) coordinate (ab) -- (b);
\draw (b) -- ++(27:1) coordinate (ba) -- (a) -- cycle;
\draw (ab) -- (45:1) -- (ba);
\draw (c) -- ++(153:1) coordinate (cd) -- (d);
\draw (d) -- ++(117:1) coordinate (dc) -- (c) -- cycle;
\draw (cd) -- (135:1) -- (dc);
\draw (e) -- ++(-117:1) coordinate (ef) -- (f);
\draw (f) -- ++(-153:1) coordinate (fe) -- (e) -- cycle;
\draw (ef) -- (225:1) -- (fe);
\draw (g) -- ++(-27:1) coordinate (gh) -- (h);
\draw (h) -- ++(-63:1) coordinate (hg) -- (g) -- cycle;
\draw (gh) -- (-45:1) -- (hg);
\foreach \x in {a,b,c,d,e,f,g,h}{\filldraw (\x) circle (0.075);}
\foreach \x in {ab,ba,cd,dc,ef,fe,gh,hg}{\filldraw (\x) circle (0.075);}
\foreach \x in {45,135,225,315}{\filldraw[fill=white] (\x:1) circle (0.075);}
\draw[dotted] (-2.2,0) -- (2.2,0);
\draw[dotted] (0,-2.2) -- (0,2.2);
\end{tikzpicture}
\caption{The augmented independence graph $\tilde{G}(\mathcal I)$ where $\mathcal I$ is the unsatisfiable $2$-LIT $3$-SAT instance $(b\vee b\vee a)\wedge(\overline b\vee\overline b\vee a)\wedge(c\vee c\vee \overline a)\wedge(\overline c\vee\overline c\vee \overline a)$. Dotted lines separate clauses, and removing the white vertices reduces the chromatic number.}
\label{fig:SAT}
\end{figure}
\begin{proposition}The graph $G=\widetilde G(\mathcal I)$ satisfies $\Delta(G)=3m$ and $\chi(G)=\omega(G)=2m$ but $\vs_\chi(G)=\vs_\omega(G)<\ivs_\omega(G)\leq\ivs_\chi(G)$.\end{proposition}
\begin{proof}Each vertex in $G(\mathcal{I})$ has $2m-2$ neighbours in the same clause and at most $m$ other neighbours corresponding to the complementary literal, with most vertices having exactly this many. (A vertex has fewer neighbours only if it corresponds to $a$ or $\overline a$ and $m$ is not a power of $2$.) In $G$, each degree increases by $2$, and the new vertices have degree $2m-1$, so $\Delta(G)=3m$. We have already established that $\chi(G(\mathcal I))=2m$, and we may greedily extend a $2m$-colouring of $G(\mathcal I)$ to one of $G$ since each uncoloured vertex has degree $2m-1$.

By Lemma \ref{lem:removal}, we can remove a set $S$ consisting of one vertex from each clause of $G(\mathcal I)$ to reduce its chromatic number. We can extend a $(2m-1)$-colouring of $G(\mathcal I)-S$ to one of $G-S$, since each uncoloured vertex only has $2m-2$ neighbours in $V(G)\setminus S$. Thus $\vs_\chi(G)=\abs S$.

Note that each vertex in $V(G)\setminus V(G(\mathcal I))$ is in a clique of order $2m$ in $G$. There are $2\abs S$ such cliques, each meeting exactly one other. Thus the only way to remove $\abs{S}$ vertices from $G$ and reduce the clique number is to remove one vertex from each of the $\abs S$ intersections, \ie one vertex from each clause of $G(\mathcal I)$. However, since $\mathcal I$ is unsatisfiable, there is no independent set of this form. Thus $\ivs_\chi(G)\geq\ivs_\omega(G)>\abs S$.
\end{proof}

\section{Wrapping up}\label{sec:wrapup}

\begin{proof}[Proof of Theorem~\ref{thm:summary}]
    For the first part, by~\cite[Theorem~1]{ABKM22} we know $f(\Delta)\le \Delta$.
    For $\Delta \le 10$, we have constructed graphs with $\vs_\chi(G)<\ivs_\chi(G)$ and $\chi(G)=\Delta(G)-1$ for $\Delta \in \{3,4\}$ in Section~\ref{sec:examples_with_vs_omega=vs_chi} (Proposition~\ref{prop:constr1}) and for $5 \le \Delta \le 10$ in Proposition~\ref{prop:chi+sqrtchi_bound}. The construction for $\Delta=5$, which is the minimal case giving a negative answer to Question \ref{problem-half}, has been depicted separately in Figure~\ref{fig:5432graph}.

    The second part is due to our construction in Proposition~\ref{prop:chi+sqrtchi_bound}.
    A small computation verifies that for $\chi=\Delta-k_{\Delta}$ (note that $k_{\Delta}\ge 1$ when $\Delta>10$), we have $\chi+ \ceilfrac{\chi}{\ceil{\sqrt \chi}}-2 \le \Delta.$
    The latter is equivalent to showing that 
    \[\Delta - k_{\Delta} \le (k_\Delta +2)\bceil{\sqrt{\Delta - k_{\Delta}}}.\]
    If $\sqrt{\Delta - k_{\Delta}} \le k_\Delta +2,$ this is clear.
    By definition of $k_{\Delta}$ we have $ \Delta < (k_\Delta +2)(k_\Delta +3)$ and so the statement is again true when $\ceil{\sqrt{\Delta - k_{\Delta}}}\ge k_\Delta +3$.
    
    For every $\Delta$ for which $(k+1)(k+2) \le \Delta \le k^2+4k+1$ for some positive integer $k\ge 1$, and there are infinitely many such $\Delta$, we have $(k+1)^2<\Delta-k+1 \le (k+1)(k+2).$
    This implies that for $\chi=\Delta-k_{\Delta}+1$ we have
    $\chi+ \ceilfrac{\chi}{\ceil{\sqrt \chi}}-2 \le \Delta$ since $\chi \le (k_\Delta+1)\ceil{\sqrt \chi}.$
    
    The third part is proven in Theorem~\ref{thr:main}.
\end{proof}

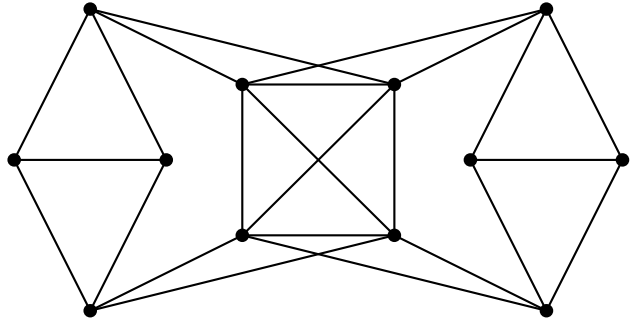
\begin{figure}[ht]
    \centering
\begin{tikzpicture}[thick]
\draw (-2,0) -- (-1,-2) -- (-2,-4) -- (0,-3);
\draw (-2,-4) -- (2,-3) -- (4,-4) -- (5,-2) -- (4,0) -- (3,-2) -- (4,-4);
\draw (3,-2) -- (5,-2);
\draw (4,0) -- (0,-1);
\draw (4,0) -- (2,-1);
\draw (-2,0) -- (0,-1);
\draw (-2,0) -- (2,-1);
\draw (-2,0) -- (-3,-2) -- (-2,-4);
\draw (-3,-2) -- (-1,-2);
\draw (0,-1) -- (0,-3);
\draw (2,-1) -- (2,-3) -- (0,-3);
\draw (0,-1) -- (2,-1) -- (0,-3) -- (4,-4);
\draw (0,-1)-- (2,-3);
\filldraw (0,-1) circle (0.075);
\filldraw (2,-1) circle (0.075);
\filldraw (2,-3) circle (0.075);
\filldraw (0,-3) circle (0.075);
\filldraw (-2,-4) circle (0.075);
\filldraw (4,-4) circle (0.075);
\filldraw (4,0) circle (0.075);
\filldraw (-2,0) circle (0.075);
\filldraw (-1,-2) circle (0.075);
\filldraw (-3,-2) circle (0.075);
\filldraw (3,-2) circle (0.075);
\filldraw (5,-2) circle (0.075);
\end{tikzpicture}
    \caption{A graph $G$ with $\Delta(G)=5, \chi(G)=4, \ivs_\chi(G)=3$ and $\vs_\chi(G)=2$, using the general construction from Figure \ref{fig:constructionD-o(D)}}
    \label{fig:5432graph}
\end{figure}

\paragraph{Acknowledgement.} This work was carried out while SC was affiliated with Warwick University and the Institute for Basic Science (South Korea), JH with Oxford University, and RJK with Radboud University Nijmegen.

\paragraph{Open access statement.} For the purpose of open access,
a CC BY public copyright license is applied
to any Author Accepted Manuscript (AAM)
arising from this submission.

\bibliographystyle{abbrv}
\bibliography{references}

\begin{thebibliography}{1}

\bibitem{ABKM22}
S.~Akbari, A.~Beikmohammadi, S.~Klav\v{z}ar, and N.~Movarraei.
\newblock On the chromatic vertex stability number of graphs.
\newblock {\em European J. Combin.}, 102:Paper No. 103504, 7, 2022.

\bibitem{BMR05}
B.~Farzad, M.~Molloy, and B.~Reed.
\newblock {$(\Delta-k)$}-critical graphs.
\newblock {\em J. Combin. Theory Ser. B}, 93(2):173--185, 2005.

\bibitem{Hax01}
P.~E. Haxell.
\newblock A note on vertex list colouring.
\newblock {\em Combin. Probab. Comput.}, 10(4):345--347, 2001.

\bibitem{Karp}
R.~M. Karp.
\newblock Reducibility among combinatorial problems.
\newblock In {\em Complexity of computer computations ({P}roc. {S}ympos., {IBM}
  {T}homas {J}. {W}atson {R}es. {C}enter, {Y}orktown {H}eights, {N}.{Y}.,
  1972)}, pages 85--103, 1972.

\bibitem{King11}
A.~D. King.
\newblock Hitting all maximum cliques with a stable set using lopsided
  independent transversals.
\newblock {\em J. Graph Theory}, 67(4):300--305, 2011.

\bibitem{MR14}
M.~Molloy and B.~Reed.
\newblock Colouring graphs when the number of colours is almost the maximum
  degree.
\newblock {\em J. Combin. Theory Ser. B}, 109:134--195, 2014.

\bibitem{PPR16}
D.~Paulusma, C.~Picouleau, and B.~Ries.
\newblock Reducing the clique and chromatic number via edge contractions and
  vertex deletions.
\newblock In {\em Combinatorial optimization}, volume 9849 of {\em Lecture
  Notes in Comput. Sci.}, pages 38--49. Springer, [Cham], 2016.

\bibitem{Staton80}
W.~Staton.
\newblock Edge deletions and the chromatic number.
\newblock {\em Ars Combin.}, 10:103--106, 1980.

\bibitem{Tov84}
C.~A. Tovey.
\newblock A simplified {NP}-complete satisfiability problem.
\newblock {\em Discrete Appl. Math.}, 8(1):85--89, 1984.

\end{thebibliography}

\end{document}